\documentclass[12pt]{article}
\usepackage{amsmath}
\usepackage{amsfonts}
\usepackage{amsthm}
\usepackage{amssymb}
\usepackage{tikz}
\usepackage{graphicx}
\usepackage{enumerate}
\usepackage[utf8]{inputenc}

\usepackage[left=2.75cm, right=2.75cm,]{geometry}
\usepackage[colorlinks=true,citecolor=blue]{hyperref}

\usepackage{comment}
\usetikzlibrary{calc, intersections}

\newtheorem{theorem}{Theorem}[section]
\newtheorem{lemma}[theorem]{Lemma}
\newtheorem{prop}[theorem]{Proposition}
\newtheorem{corollary}[theorem]{Corollary}
\newtheorem{Q}[theorem]{Question}

\theoremstyle{definition}
\newtheorem{eg}[theorem]{Example}

\theoremstyle{remark}
\newtheorem{claim}{Claim}[theorem]

\newtheorem{remark}[claim]{Remark}

\DeclareMathOperator{\conv}{conv}

\usepackage[normalem]{ulem} 
\usepackage{bm}
\usepackage{algorithm}
\usepackage{algpseudocode}

\newcommand{\rhalf}[1]{H_{#1}^r}
\newcommand{\lhalf}[1]{H_{#1}^l}

\newcommand{\Line}[1]{\overline{#1}}
\newcommand{\mc}{\mathcal}
\newcommand{\mb}{\mathbb}
\newcommand{\im}{\mathfrak{i}}
\newcommand{\df}[1]{\textbf{\textit{\color{cyan!10!black} #1}}}

\newcommand{\eps}{\varepsilon}

\DeclareMathOperator{\vis}{vis}
\DeclareMathOperator{\cone}{cone}
\DeclareMathOperator{\sign}{sign}
\DeclareMathOperator{\id}{id}

\title{Inscribable order types}

\date{} 

\author{Michael Gene Dobbins\footnote{%
Department of Mathematical Sciences, Binghamton University, Binghamton, NY, USA.\newline  \texttt{mdobbins@binghamton.edu},\newline 
support from KAIST Advanced Institute for Science-X (KAI-X)} %
\ and Seunghun Lee\footnote{%
Einstein Institute of Mathematics, Hebrew University, Jerusalem, Israel.\newline \texttt{seunghun.lee@mail.huji.ac.il}}%
}

\begin{document}

\maketitle

\begin{abstract}
We call an order type \textit{inscribable} if it is realized by a point configuration where all extreme points are all on a circle. In this paper, we investigate inscribability of order types. 
We first construct an infinite family of \textit{minimally uninscribable} order types. The proof of uninscribability mainly uses Möbius transformations and the Frantz ellipse. 
We further show that every simple order type with at most 2 interior points is inscribable, and that the number of such order types is $\Theta(\frac{4^n}{n^{3/2}})$. 
We also suggest open problems around inscribability.
\end{abstract}

\section{Introduction}\label{section_intro}
The \df{orientation} of an ordered triple of points in the plane is positive when the points appear in counterclockwise order around their convex hull, is negative when they appear in clockwise order, and is zero when they are collinear. 
A pair of finite point sets $P,Q \subset \mb{R}^2$, which we call configurations, have the same order type when there is a bijection 
that preserves the orientation of each triple.
This defines an equivalence relation.  An \df{order type} $\omega$ is an equivalence class of this relation, and we say an element of $\omega$ \df{realizes} $\omega$.
An order type is \df{simple} when there are no collinear triples. 
The \df{extreme} points of a configuration $P$ are the vertices of the convex hull of $P$, i.e., the points of $P$ that can be isolated by a line, and the other points of $P$ are \df{interior} points.
We say a configuration is \df{inscribed} when its extreme points are all on a circle, 
and we say a configuration (or its order type) is \df{inscribable} when it has the same order type as an inscribed configuration, otherwise we say it is \df{uninscribable}. 

The analogous notion of inscribability for polytopes is an active area of research with a long history.
In 1832 Jakob Steiner asked if every 3-dimensional polytope is combinatorially equivalent to a polytope with vertices on a sphere \cite{steiner1832systematische}.
This question was answered negatively about a hundred years later when Ernst Steinitz gave the first example of a polytope that is not inscribable \cite{steinitz1928isoperimetrische}, 
and research around inscribability of polytopes continues to this day 
\cite{akopyan2021beauty,doolittle2020combinatorial,edelsbrunner2018random,chen2017scribability,firsching2017realizability,padrol2016six,gonska2016f,rivin1996characterization}.  

In this paper we introduce what we believe are the first examples of uninscribable configurations.
The smallest is the \df{non-Pascal configuration}; see Figure \ref{figNonPascal}.  This example is inspired by the non-Pappus matroid, which was an important example in the study of realizability of matroids and oriented matroids \cite{levi1926teilung,ringel_simple-non-Pappus}.  

The uninscribability of the non-Pascal configuration follows from Pascal's theorem, which states that if points $p_1,\dots,p_6$ are on an ellipse in that order, then the points  
$c_1 = \Line{p_1p_5}\cap\Line{p_2p_6}$, $c_2 = \Line{p_1p_4}\cap\Line{p_3p_6}$ and $c_3  = \Line{p_2p_4}\cap\Line{p_3p_5}$ are collinear \cite[Theorem 1.4]{projective_geometry_book}. 
This extends Pappus's theorem from classical antiquity, which corresponds to the degenerate case where the ellipse is replaced by a pair of lines.

\begin{figure}
\centering
\includegraphics{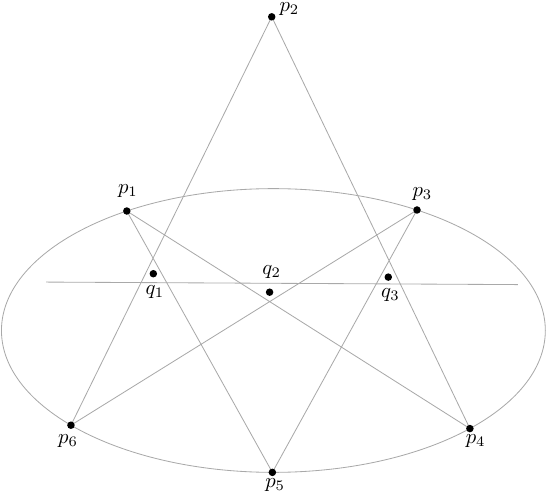}
\caption{The point set $\{p_1,\dots,p_6,q_1,q_2,q_3\}$ is a non-Pascal configuration.}
\label{figNonPascal}
\end{figure}

\begin{eg} \label{eg_Non_Pascal}
Suppose that the non-Pascal configuration could be inscribed. 
Then, the points $c_1 = \Line{p_1p_5}\cap\Line{p_2p_6}$, $c_2 = \Line{p_1p_4}\cap\Line{p_3p_6}$ and $c_3  = \Line{p_2p_4}\cap\Line{p_3p_5}$ would be collinear by Pascal's theorem, which would force the orientation of $Q = (q_1,q_2,q_3)$ to be negative, but $Q$ is positively oriented by definition of the non-Pascal configuration. \qed
\end{eg}

In addition to being the uninscribable configuration with the fewest number of points, we show that the non-Pascal configuration is minimal in the following two ways. 

\begin{theorem}\label{theorem_two_interior_inscribable}
Every simple order type with at most 2 interior points is inscribable.
\end{theorem}

We prove Theorem \ref{theorem_two_interior_inscribable} by giving an explicit construction. 
Note that any two points inside the circle can be chosen as the interior points for this construction.

\begin{prop} \label{prop_(5,3)_inscribable}
Every simple order type with at most 5 extreme points is inscribable.
\end{prop}

We would like to have a nice characterization of inscribable order types.
We may hope for a finite list of obstructions to inscribability, 
like Kuratowski's characterization of planar graphs. 
Similarly, many classes of graphs can be characterized by a finite list of forbidden minors \cite{robertson2004graph}, and this extends to many classes of matroids as well \cite{tutte1958homotopy}. 

Matroids are purely combinatorial objects defined by axioms coming from properties of vector configurations in linear algebra like rank, span, and linear independence.  Oriented matroids are matroids with additional sign information corresponding to positively or negatively oriented bases, and in particular, the orientations of each tipple of points of an order type defines an oriented matroid.  However, not all matroids or oriented matroids can be realized by a vector configuration or affine point set \cite{om_book}.  

Matroids have a natural partial ordering defined by the operations of deletion and contraction, and when two matroids are related, the lower one in the poset is called a minor of the higher one.    
The class of matroids that are realizable over every field has been characterized by a finite list of forbidden minors, for example.  That is, the class of such matroids has only finitely many members that are minimal in this partial ordering, which are the obstructions to realizability in all fields \cite{tutte1958homotopy}.  On the other hand, there is no such characterization for matroids that are realizable over just the reals \cite{vamos1978missing}.

To distinguish one obstruction to inscribability from another, we need a way of ordering order types and an associated notion of minimality. 
Order types can be ordered by containment, but that is not the right ordering here, since inscribable order types are not closed under deletion. 
For example, a triangle with the non-Pascal configuration in its interior is inscribable.  
From this example, the problem with containment becomes clear; the extreme points play a special role, but additional points can become extreme points as a result of deletion. 
In this way, we are really considering a property of an order type with a distinguished subset of elements.  The appropriate ordering then is containment for both the order type and the distinguished subset.  
For $B\subseteq P \subset \mb{R}^2$, we say the pair $(P,B)$ is \df{inscribable} when there is a realization of $P$ in the unit disk such that all the points among $B$ are realized on the unit circle.  Note that we do not require $B$ to be a set of extreme points of $P$, but this must be the case for $(P,B)$ to possibly be inscribable. 
We say $P$ is \df{minimally uninscribable} when $P$ is not inscribable, but $(P',B')$ is always inscribable provided that $P'$ is a proper subset of $P$ and $B'\subseteq P'$ is a subset of the extreme points of $P$.  

The non-Pascal configuration is not the only obstruction to inscribability. 
Indeed we construct an infinite family of distinct obstructions.

\begin{theorem} \label{theorem_main_minimally_uninscribable}
There are infinitely many minimally uninscribable order types.  
\end{theorem}


Our interest in inscribable order types originally stemmed from the problem of counting order types with relatively few interior points.
The orientation of an ordered triple $(p_1,p_2,p_3)$ can be expressed algebraically by
\[
\sign\left(\det\left[\begin{array}{ccc} p_1 & p_2 & p_3 \\ 1 & 1 & 1 \end{array}\right]\right) \in \{1,0,-1\}. 
\]
Goodman and Pollack used this fact to show an upper bound $2^{\Theta(n\log n)}$ on the number of simple labelled order types of size $n$ \cite{GP_upper_bound_ot}, which was later improved by Alon \cite{alon_upper_bound_ot}.  
Specifically, they used a theorem by Milnor and Thom from real algebraic geometry to bound the number of possible sign patterns of the determinants in terms of the number of variables and degree of the corresponding system of polynomials \cite{milnor-thom_by_milnor, milnor-thom_by_thom}  (see also \cite{warren}). 

Xavier Goaoc suggested in a personal communication that, for order types with a large number of extreme points, it may be possible to realize the extreme points on a fixed algebraic curve, and this could give a better bound on the number of such order types than the Goodman-Pollack and Alon bounds. Goaoc's suggestion led the authors to consider the following question.

\begin{Q} \label{questionFewInLowDeg}
For each $k$ is there an algebraic curve or zero-locus $\gamma_k$ such that every simple order type with at most $k$ interior points has a realization where all the extreme points are on $\gamma_k$? If so, then what is the smallest degree of such a set $\gamma_k$?
\end{Q}

For each nonnegative integer $k$, let $d_k$ be the minimum nonnegative integer such that for every simple order type $\omega$ with at most $k$ interior points there is a degree $d_k$ curve $\gamma$ such that $\omega$ has a realization where all the extreme points are on $\gamma$. If there are no such curves, then we set $d_k= \infty$. The following question is a weaker version of Question \ref{questionFewInLowDeg}.

\begin{Q} \label{questionFewInLowDeg_weaker}
Is $d_k$ finite for every nonnegative integer $k$? If so, how does $d_k$ grow?
\end{Q}

In Lemma \ref{lemma_uninscribable_circle_conic} we show that an order type is inscribable if and only if its extreme points can be realized on a conic, so in the special case of Question \ref{questionFewInLowDeg} for degree 2 curves, we may assume that $\gamma_k$ is a circle.
Hence, Theorem \ref{theorem_two_interior_inscribable} answers Questions \ref{questionFewInLowDeg} and \ref{questionFewInLowDeg_weaker} in the case where $k=2$; in particular $d_2=2$. Also the non-Pascal configuration shows that $d_k \geq 3$ when $k \geq 3$.

\medskip 

Simple point sets with at most one interior point are called \df{conowheels}. The number of such order types was counted 
by giving a 1-1 correspondence between the order types of conowheels of size $n+1$ and 
2-colored self-dual necklaces with $2n$ beads \cite{self_dual_necklace-brouwer, self_dual_necklace-palmer, one_point_pilz_welzl}  
\cite[Chapter 6.3]{grunbaum_polytope_book} which is asymptotically $\Theta(\frac{2^n}{n})$, see Theorem \ref{theorem_conowheel}.
Using this combinatorial characterization, it is not hard to show that conowheels are inscribable. 
Theorem \ref{theorem_two_interior_inscribable} extends this to simple order types with two interior points.

Combining proof ideas for Theorem \ref{theorem_two_interior_inscribable} with MacMahon's counting result on plane partitions (see \cite{macmahon_counting_book} and \cite[p.545]{handbook_enumerative_combinatorics}), we provide an estimate for the number of simple planar order types with at most 2 interior points which is $\Theta(\frac{4^n}{n^{3/2}})$, see Theorem \ref{theorem_counting_2_interior}.

\medskip

This manuscript is organized as follows. 
In Section \ref{section_construction}, we prove Theorem \ref{theorem_main_minimally_uninscribable} by constructing an infinite family of minimally uninscribable configurations. We additionally show that configurations in the family do not contain each other as subconfigurations. 
In Section \ref{Section_few_points}, we consider order types with few interior or extreme points, and prove Theorem \ref{theorem_two_interior_inscribable}
and Proposition \ref{prop_(5,3)_inscribable} along with Lemma \ref{lemma_uninscribable_circle_conic}.
We also asymptotically count the number of simple order types with at most 2 interior points. 
In Section \ref{section_concluding}, we give some remarks and suggest open problems. 

We use the following notation.
$[k] = \{1,\dots,k\}$. 
$\Line{ab}$ is the line through points $a$ and $b$. 
$[a,b]$ is the line segment with endpoints $a$ and $b$.
$\mb{S}^1$ is the unit circle in the complex plane, 
and $\mb{D}$ is the unit disk in the complex plane. 


\section{Minimally uninscribable order types} 
\label{section_construction}

In this section we first prove Theorem \ref{theorem_main_minimally_uninscribable} in Subsection \ref{subsection_k-gon} by explicitly constructing a sequence of minimally uninscribable order types $P_n$.  We then show in Subsection \ref{subsection_incomparable} that none of these order types are related by containment.  Finally, we present a more involved procedure in Subsection \ref{subsection_k-star} for constructing many more uninscribable order types.

\subsection{An infinite family of minimally uninscribable order types} \label{subsection_k-gon}

We first define the uninscribable point configurations $P_n$.
Fix $n \geq 3$, and 
let $a_1,c_1,\dots,a_n,c_n$ be $2n$ points spaced evenly around the unit circle in counterclockwise order.
Let $L_i$ be the directed line that bisects the angle from the directed line $\Line{a_ia_{i+1}}$ to $\Line{c_ic_{i+1}}$, and that is offset to the right from $\Line{a_ia_{i+1}} \cap \Line{c_ic_{i+1}}$ by some $\eps > 0$ sufficiently small that the points $c_{i}$ and $a_{i+1}$ are to still the right of $L_i$.
Let $b_i = L_i\cap L_{i-1}$, and let $d_i$ be a point in the triangle bounded by 
$L_i = \Line{b_ib_{i+1}}$, $\Line{a_ia_{i+1}}$, and $\Line{c_ic_{i+1}}$.
Note that we treat indices modulo $n$, so in particular $L_0 = L_n$. 
Let $P_n = \{a_i,b_i,c_i,d_i: i \in [n]\}$; see Figure \ref{figPn}.

\begin{figure}[ht]
\begin{center}
\begin{tikzpicture}[scale=0.7,rotate = 90]
\def\r{.07}
\def\s{5.3}
\def\t{3.4}
\def\u{2.4}
\draw[fill]
(0:\s) circle (\r) --
(120:\s) circle (\r) --
(240:\s) circle (\r) -- (0:\s)
(30:\t) circle (\r) --
(150:\t) circle (\r) --
(270:\t) circle (\r) -- (30:\t)
(90:\t) circle (\r) --
(210:\t) circle (\r) --
(330:\t) circle (\r) -- (90:\t)
(60:\u) circle (\r)
(180:\u) circle (\r)
(300:\u) circle (\r)
;
\def\s{5.7}
\def\t{3.8}
\def\u{3}
\path 
(330:\t) node {$a_1$}
(90:\t)  node {$a_2$}
(210:\t) node {$a_3$}
(0:\s)   node {$b_1$}
(120:\s) node {$b_2$}
(240:\s) node {$b_3$}
(30:\t)  node {$c_1$}
(150:\t) node {$c_2$}
(270:\t) node {$c_3$}
(60:\u)  node {$d_1$}
(180:\u) node {$d_2$}
(300:\u) node {$d_3$}
; 

\end{tikzpicture}
\caption{The configuration $P_3$.}
\label{figPn}
\end{center}
\end{figure}
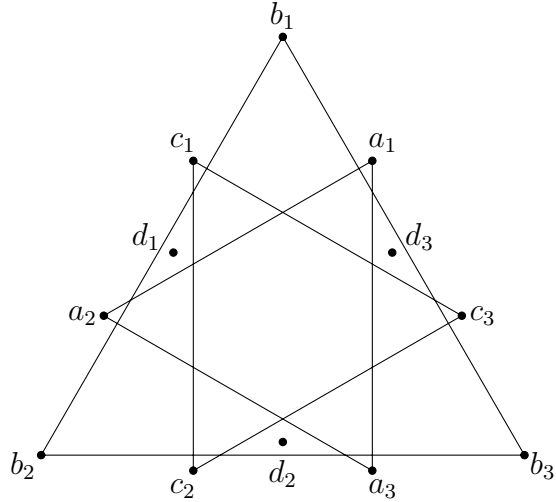

We will show that $P_n$ is uninscribable using Möbius transformations of the unit circle in the complex plane $\mb{S}^1 \subset \mb{C}$.
In general a Möbius transformation of the complex plane is a map of the form
\[ f(z) = \frac{az+b}{cz+d}. \] 
Note that a Möbius transformation is a bijection of the one point compactification of the complex plane $\overline{\mb{C}}$, 
and that a Möbius transformation is determined by the image of 3 points. 

Here we will mainly use Möbius transformations of the unit circle, i.e., transformations $\phi:\mb{S}^1 \to \mb{S}^1$.
These transformations are of the form 
\[ \phi(z) = \phi_{a,\theta}(z) = e^{\im\theta}\frac{z-a}{\overline{a}z-1}, \] 
for $a \in \mb{D}^\circ = \{a \in \mb{C} : |a|<1\}$ and $\theta \in \mb{R}$. 
In the case where $\theta = 0$ (or $\theta = 2\pi\mb{Z}$), $\phi$ is the map that projects a point on $\mb{S}^1$ through the point $a$ to $\mb{S}^1$.  That is, $[z,\phi_{a,0}(z)]$ is a segment through the point $a$.  In general, for $a,\theta$ fixed, the collection of segments $[z,\phi_{a,\theta}(z)]$ are tangent to an ellipse in $\mb{D}$, which we will call the \df{Frantz ellipse} of the map $\phi$ \cite{frantz2004conics}.  In the case where $\theta = 0$, we regard the point $a$ as a degenerate ellipse.  
Note that if $\phi$ has a fixed point, then the Frantz ellipse will be tangent to the unit circle at that fixed point.

We first give a simpler proof of uninscribability in the case $n=3$.  This is not needed to understand the general case, but is of interest as an alternate proof, and may provide the reader insight that will be helpful for understanding the general case.

\begin{figure}[ht]
\centering
\includegraphics{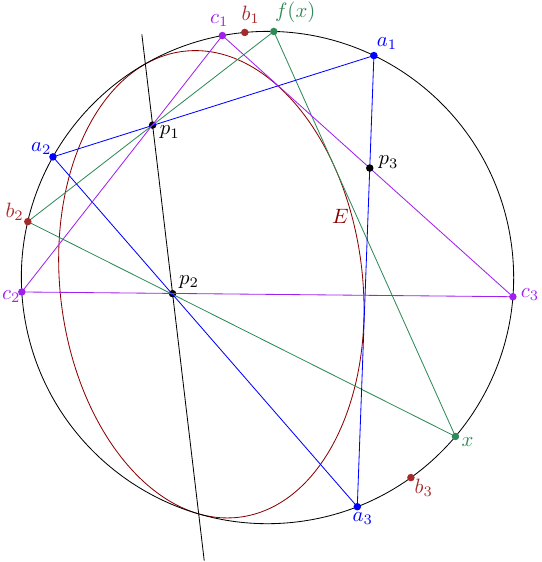}
\caption{We use the Frantz ellipse to show that $P_3$ is not inscribable.}
\label{figP3}
\end{figure}

\begin{theorem}\label{theoremP3Uninscribable}
$P_3$ is uninscribable.
\end{theorem}

\begin{proof}
Suppose $P_3$ were inscribable and consider an inscribed realization.  
Let $p_k = [a_ka_{k+1}]\cap[c_kc_{k+1}]$ with indices mod 3, 
and let $f_k : \mb{S}^1 \to \mb{S}^1$ be the projection through $p_k$, 
and $f = f_1\circ f_2$. 
Note that $f$ and the $f_k$ are all Möbius transformations.
Let $E$ be the Frantz ellipse of $f$, 
and let $x = f_2(b_2)$; see Figure \ref{figP3}.
Since $f_2$ is an involution, $f(x) = f_1(b_2)$. 

By the construction of $P_3$, the point $d_2$ is to the left of the line $L_2 = \Line{b_2b_{3}}$ directed from $b_2$ to $b_3$,  
and $p_2$ is a vertex of the triangle bounded by the lines $L_2$, $\Line{a_2a_{3}}$, and $\Line{c_2c_{3}}$, which contains $d_2$, 
so the point $p_2$ must also be on the left of $L_2$. 
Also, the triple $(b_2,c_2,c_3)$ is positively oriented, so 
the triple $(b_2,p_2,c_3)$ is positively oriented, 
so the points 
$b_3$, $p_2$, $c_3$ appear in that order counterclockwise around the point $b_2$. 

Similarly, the point $d_1$ forces $p_1$ to be to the left of the line $L_1 = \Line{b_1b_{2}}$ directed from $b_1$ to $b_2$, which together with the above and the fact that $(b_2,c_3,a_1)$ is positively oriented, implies that the points 
$b_3$, $p_2$, $c_3$, $a_1$, $p_1$, $b_1$ appear in that order counterclockwise around the point $b_2$.
Therefore, the points 
$b_3$, $x$, $f(x)$, $b_1$ appear in that order counterclockwise around the point $b_2$, 
so $x$ and $f(x)$ are both to the right of the directed chord $[b_3,b_1]$ from $b_3$ to $b_1$

Also, $[b_3,b_1]$ is to the right of the directed chord $[a_3,c_1]$, so is $[x,f(x)]$ to the right of $[a_3,c_1]$. 
Since $a_1=f(a_3)$ and $c_1 = f(c_3)$, 
the chords $[a_3,a_1]$, $[x,f(x)]$, and $[c_3,c_1]$ are each tangent to $E$ in that order counterclockwise, and all three chords are to the right of $[a_3,c_1]$, so $p_3$ must be to the right of the directed chord $[x,f(x)]$, which implies that $p_3$ must also be to the right of $[b_3,b_1]$, but the point $d_3$ forces $p_3$ to be to the left of $[b_3,b_1]$, which is a contradiction.
\end{proof}

We now prove two lemmas, which we will need for Theorem \ref{theorem_main_minimally_uninscribable}.

\begin{lemma}\label{LemmaTwoPolygons}
Let $p_k$ for $k \in \{1,\dots,n\}$ be $n \geq 3$ points inside the unit circle in convex position. 
Then, there are at most 2 inscribed $n$-gons with one of the points $p_k$ on each edge.   
Moreover, if $G_\mathrm{a}$ and $G_\mathrm{c}$ are two such $n$-gons, then $G_\mathrm{a}$ and $ G_\mathrm{c}$ are the only inscribed $n$-gons that contain $G_\mathrm{a}\cap G_\mathrm{c}$. 
\end{lemma}

\begin{proof}[Proof of Lemma \ref{LemmaTwoPolygons}]
Let us assume without loss of generality that the points $p_k$ are ordered counterclockwise. 
Hence, the edge of $G_\mathrm{a}$ through $p_k$ is adjacent to the edges through $p_{k+1}$ and $p_{k-1}$ where indices are regarded mod $n$, so in particular $p_{n+1} = p_1$; see Figure \ref{figureACP}.   

\begin{figure}[ht]
\begin{center}
 \includegraphics{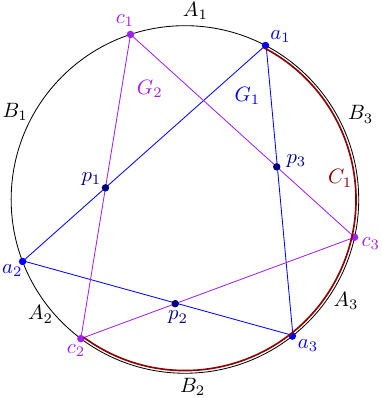}
 \caption{The circle divided into arcs by the polygons $G_1$ and $G_2$.}
\label{figureACP}
\end{center}
\end{figure}

Let $a_{k}$ and $c_k$ be the respective vertices of $G_\mathrm{a}$ and $G_\mathrm{c}$ adjacent to the edges through $p_k$ and $p_{k-1}$.  Let us also assume without loss of generality that $c_1$ appears on the counterclockwise arc from $a_1$ to $a_2$.  By induction, this implies that $c_k$ appears on the counterclockwise arc from $a_k$ to $a_{k+1}$, since the edge $[c_k,c_{k+1}]$ crosses the edge $[a_{k},a_{k+1}]$ at $p_{k}$.  

\begin{figure}[ht]
\centering
\includegraphics{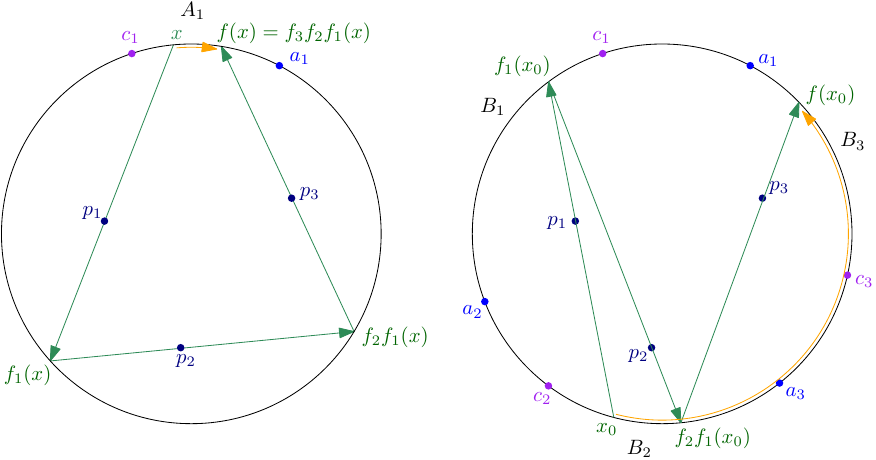}
\caption{The maps $f_k$ and the point $x_0$.}
\label{figureF}
\end{figure}

Let $f_k : \mb{S}^1 \to \mb{S}^1$ be the projection through $p_k$, and let $f = f_n\circ \cdots \circ f_1$; see Figure \ref{figureF}.  
Let $\overline{\mb{R}}$ be the one point compactification of $\mb{R}$, and 
let $g : \mb{S}^1 \to \overline{\mb{R}}$ be the restriction to the unit circle of the Möbius transformation where $g(a_1) = 0$, $g(c_1) = \infty$, and $g(x_1) = 1$ where $x_1$ is the midpoint of the counterclockwise arc along $\mb{S}^1$ from $a_1$ to $c_1$.  
Let $h = g \circ f \circ g^{-1} : \overline{\mb{R}} \to \overline{\mb{R}}$.
Each $f_k$ is a Möbius transformation, so $h$ is a Möbius transformation. 
Also, $h(0) = 0$, $h(\infty)=\infty$, and $h$ preserves the upper half-plane. 
Hence, $h(x) = r x$ for some $r > 0$. 
\begin{claim}\label{claimRSmall}
$r  < 1$. 
\end{claim}

Let $A_k$ be the counterclockwise arc from $a_k$ to $c_k$, and let $B_k$ be the counterclockwise arc from $c_k$ to $a_{k+1}$, and let $C_k = \mb{S}^1 \setminus \overline{A_k\cup B_k\cup A_{k+1}}$
be the counterclockwise arc from $c_{k+1}$ to $a_{k}$; see Figure \ref{figureACP}. 
Observe that $f_k$ sends $A_k$ to $A_{k+1}$ and vice-versa, and sends $B_k$ to $C_{k}$ and vice-versa.    
Observe also that $g$ sends $A_1$ to the positive reals, and sends the counterclockwise arc from $c_1$ to $a_1$ to the negative reals. 

Observe that $p_2$ is separated from $B_1$ by the chords $[a_1,a_2]$ and $[c_1,c_2]$, which meet at $p_1$.  
Let $x_0$ be a point on the counterclockwise arc $B_2 \subset C_1$ from $c_2$ to $a_3$. 
Then, $f_1(x_0) \in B_1 \subset C_2$,
and by induction $f_k\cdots f_1(x_0) \in B_k \subset C_{k+1}$.  
In particular, $f(x_0) = f_n\cdots f_1(x_0) \in B_n$, which is counterclockwise of $x_0 \in B_1$ along the counterclockwise arc from $c_1$ to $a_1$. 
Hence, $g(x_0) < g(f(x_0)) < 0$, so 
$g(x_0) < h(g(x_0)) = r g(x_0) < 0$, so Claim \ref{claimRSmall} holds, i.e., $r < 1$.  
 
Suppose for the sake of contradiction that there is an inscribed $n$-gon $G_\mathrm{x}$ distinct from $G_\mathrm{a}$ and $G_\mathrm{c}$ with one of the points $p_k$ on each edge.  Let $x_1$ be the vertex of $G_\mathrm{x}$ adjacent to the edges through $p_1$ and $p_n$.  Then, $x_2 = f_1(x_1), x_3 = f_2(x_2),\dots, x_1 = f_n(x_n) = f(x_1)$, but since $r \neq 1$, the only fixed points of $f$ are $a_1$ and $c_1$, so $f(x_1) \neq x_1$, which is a contradiction.

\smallskip

Next, we show the second part of the lemma.

Suppose for the sake of contradiction that there is an inscribed $n$-gon $G_\mathrm{b}$ that contains $G_\mathrm{a} \cap G_\mathrm{c}$. 
We may assume without loss of generality that $G_\mathrm{b}$ has a vertex $b_1$ in the half open arc $A_1\setminus \{a_1\}$;
otherwise we may relabel the vertices of accordingly.
Let $b_1,b_2,\dots,b_n$ be the vertices of $G_\mathrm{b}$ in counterclockwise order.
We will show that, with this labeling, $G_\mathrm{b}$ must be $G_\mathrm{c}$.

For $x \in \mb{S}^1$, let $\theta(x) \in [0,2\pi)$ such that $x=e^{\im\theta(x)}$, and let us also assume without loss of generality that $\theta(b_1)=0$. 
Otherwise, we could rotate the configuration by $-\theta(b_1)$.
Then, $\theta(b_1)<\theta(b_2)<\dots<\theta(b_n)<2\pi$.

Let $L(\theta)$ be the directed chord extending from $e^{\im\theta}$ that is tangent to $G_\mathrm{a} \cap G_\mathrm{c}$ on the left side of $L(\theta)$, and let $\nu(\theta) \in (\theta,\theta+2\pi)$ such that $e^{\im\nu(\theta)}$ is the other point where $L(\theta)$ meets the unit circle.
Let $\vis(\theta)$ be the set of angles of the points on the unit circle that are visible from $e^{\im\theta}$ in $\mb{D}\setminus (G_\mathrm{a} \cap G_\mathrm{c})^\circ$.
That is, $\psi \in \vis(\theta) \subset (\theta-2\pi,\theta+2\pi)$ when the chords $[e^{\im\theta},e^{\im\tau}]$ are disjoint from $(G_\mathrm{a} \cap G_\mathrm{c})^\circ$ for $\tau$ in $[\theta,\psi]$ or $[\psi,\theta]$.  Note that $\vis(\theta)$ and $\nu(\theta)$ are not restricted to $[0,2\pi)$. 

\begin{remark}\label{remarkNuInc}
If $\theta < \psi$, then $\nu(\theta) < \nu(\psi)$.
\end{remark}

\begin{remark}\label{remarkNuVis}
$\nu(\theta) = \max(\vis(\theta))$.
\end{remark}

Since $b_1$ is separated from $G_\mathrm{a}$ by $\overline{a_1a_2}$, which passes through $p_1$, the segment $[b_1,p_1]$ is disjoint from $G_\mathrm{a}^\circ$.  
Either $b_1 \in A_1^\circ$, in which case $f_1(b_1) \in A_2^\circ$, or $b_1 = c_1$, in which case $f_1(b_1)=c_2$.
In the first case, $f_1(b_1)$ is separated from $G_\mathrm{c}$ by $\overline{c_1c_2}$, 
and in the second case, $f_1(b_1)$ is on the line $\overline{c_1c_2}$.
In either case, $[f_1(b_1),p_1]$ is disjoint from $G_\mathrm{c}^\circ$.
Hence, the directed chord $[b_1,f_1(b_1)] = [b_1,p_1]\cup[p_1,f_1(b_1)]$ passes $G_\mathrm{a} \cap G_\mathrm{c}$ tangentially on the left, 
so $[b_1,f_1(b_1)] = L(\theta(b_1))$ and $\theta f_1(b_1) = \nu\theta(b_1)$.
Likewise, 
$[f_{k-1}\cdots f_1(b_1),f_{k}\cdots f_1(b_1)]_\mb{C} = L(\nu^{k-1}\theta(b_1))$
and $\nu^k\theta(b_1) = \theta(f_k\cdots f_1(b_1))$ 
by induction.

Assuming $\nu^{k-1}\theta(b_1) \geq \theta(b_{k})$ by induction, we have $\nu^{k}\theta(b_1) \geq \nu\theta(b_{k})$ by Remark \ref{remarkNuInc}.
Since $b_{k+1}$ is visible from $b_{k}$ in $\mb{D}\setminus (G_\mathrm{a} \cap G_\mathrm{c})^\circ$, we have $\nu\theta(b_{k}) \geq \theta(b_{k+1})$ by Remark \ref{remarkNuVis}, 
so $\nu^{k}\theta(b_1) \geq \theta(b_{k+1})$.  
Furthermore, $\nu^{k}\theta(b_1) = \theta(b_{k+1})$ only if $b_{k+1} = f_k(b_{k}) = f_k\dots f_1(b_1)$ by induction.

Since $r < 1$ by Claim \ref{claimRSmall}, the point $f(b_1) = e^{\im\nu^n\theta(b_1)}$ is 
either clockwise of $b_1$ on $A_1$ in the case where $b_1 \in A_1^\circ$, 
or $f(b_1)=b_1$ in the case where $b_1 = c_1$.
Hence, $\nu^n\theta(b_1) \leq 2\pi$ with $\nu^n\theta(b_1) = 2\pi$ only if $b_1 = c_1$.

Since $b_1$ is visible from $b_n$, we have $\nu\theta(b_n) \geq 2\pi$.  Hence, $\nu\theta(b_n) \geq 2\pi \geq \nu^{n}\theta(b_1) \geq \nu\theta(b_n)$, so we must have equality in each case, which implies that $b_1=c_1$ and $b_k = f_{k-1}(b_{k-1}) = f_{k-1}(c_{k-1}) = c_k$ for each $k \in \{1,\dots,n\}$, so $G_\mathrm{b} = G_\mathrm{c}$. 
\end{proof}

\begin{figure}[ht]
\centering
\includegraphics{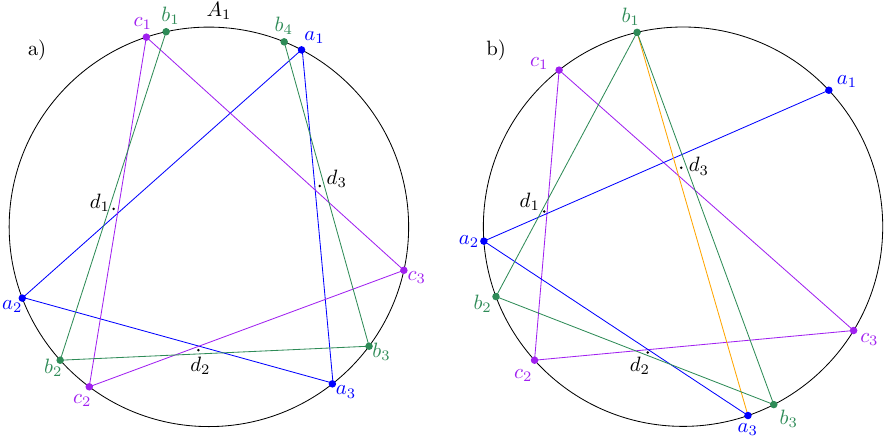}
\caption{(a) An inscribed realization of $P_3\setminus\{b_1\}$. (b) an inscribed realization of $P_3\setminus\{a_1\}$ and $P_3\setminus\{d_n\}$}
\label{figNoabd}
\end{figure}

\begin{lemma}\label{lemmaOnePointMissing}
$P_n\setminus\{a_1\}$, $P_n\setminus\{b_1\}$, and $P_n\setminus\{d_n\}$ are each inscribable.
\end{lemma}

\begin{proof}
Let points $a_k$ and $c_k$ be as in the construction of $P_n$, and let $p_k$ and $f_k$ be as in the proof of Lemma \ref{LemmaTwoPolygons}. 

Let us start with $P_n\setminus\{b_1\}$; see Figure \ref{figNoabd} (a).
Let $b_1$ be a point on the counterclockwise arc $A_1$ from $a_1$ to $c_1$, and let $b_k = e^{-\im\eps}f_{k-1}(b_{k-1})$ for $k \in \{2,\dots,n+1\}$ with $\eps>0$ sufficiently small that $b_k$ is still on the counterclockwise arc $A_k$ from $a_k$ to $c_k$.
Note that $p_k$ is to the left of $\overline{b_{k}b_{k+1}}$, 
and let $d_k$ be in the triangle to the left of $\Line{b_{k}b_{k+1}}$, $\Line{a_{k+1}a_{k}}$, and $\Line{c_{k+1}c_{k}}$ for $k \in \{1,\dots,n\}$ with the identifications $a_{n+1}=a_1$ and $c_{n+1}=c_1$, but with $b_1$ and $b_{n+1}$ distinct points.
It is tedious but straightforward to check that this configuration with $b_1$ and $b_{n+1}$ removed is an inscribed realization of $P_n\setminus\{b_1\}$. 

Next, we inscribe $P_n\setminus\{a_1\}$ and $P_n\setminus\{d_n\}$; see Figure \ref{figNoabd} (b). 
Let $b_1$ be a point on the arc $A_1$ that is sufficiently close to $a_1$ that $\Line{a_nb_1}$ intersects $\Line{a_1a_2}$ to the right of $\Line{c_nc_1}$.
Let $b_k = e^{-\im\eps}f_{k-1}(b_{k-1})$ for $k \in \{2,\dots,n\}$ with $\eps>0$ sufficiently small that $b_k \in A_k$.
Let $d_k$ be in the triangle to the left of $\Line{b_{k}b_{k+1}}$, $\Line{a_{k+1}a_{k}}$, and $\Line{c_{k+1}c_{k}}$ for $k \in \{1,\dots,n-1\}$. 
Since the points $b_1,c_1,a_n,b_n,c_n$ appear counterclockwise around the circle in that order, there is a triangular region to the left of $\Line{b_nb_1}$, $\Line{b_1a_n}$, and $\Line{c_1c_n}$.
By our choice of $b_1$ sufficiently close to $a_1$, the line $\Line{a_1a_2}$ subdivides this triangular region.
Let $d_n$ be in the region to the left of $\Line{b_nb_1}$, $\Line{a_1a_2}$, $\Line{b_1a_n}$, and $\Line{c_1c_n}$. 
This configuration with $a_1$ removed is an inscribed realization of $P_n\setminus\{a_1\}$. 
Additionally, this configuration with $d_n$ removed is an inscribed realization of $P_n\setminus\{d_n\}$. \end{proof}

\begin{proof}[Proof of Theorem \ref{theorem_main_minimally_uninscribable}]
If $P_n$ were inscribable, then for the inscribed realization, $G_\mathrm{b} = \conv(b_1,\dots,b_n)$ would be an inscribed $n$-gon containing the intersection of the analogous inscribed polygons $G_\mathrm{a}\cap G_\mathrm{c}$ whose vertices alternate around the circle, which would contradict Lemma \ref{LemmaTwoPolygons}.  Hence, $P_n$ is uninscribable.

Let $B_n$ be the set of extreme points of $P_n$.
To verify that $P_n$ is minimally uninscribable, it is enough to check that $(P_n,B_n)$ with one vertex removed is an inscribable pair.  By symmetry, we may assume that the point we removed is either $a_1$, $b_1$, or $d_n$.  Thus, $P_n$ is minimally uninscribable by Lemma \ref{lemmaOnePointMissing}.
\end{proof}

\subsection{Incomparability with respect to containment}\label{subsection_incomparable}

Since being minimally uninscribable is not defined in terms of containment of order types, Theorem \ref{theorem_main_minimally_uninscribable} does not immediately imply that none of the uninscribable order types we constructed appears as a subconfiguration of another.  This is, however, the case.

\begin{theorem}
$P_k$ is not a subconfiguration of $P_n$ for $k<n$.
\end{theorem}

The theorem follows directly from the next lemma. 

\begin{lemma}
Among the points of $P_n$, there is no set of $m$ points such the convex hull contains $m$ other points for $n > m \geq 3$. 
\end{lemma}

\begin{proof}
Suppose the lemma fails and let $q_1,\dots,q_m$ be points among $P_n$ with $m$ other points in their convex hull, which we denote by $p_1,\dots,p_m$. 
Since the points $a_i,b_i,c_i$ are all on the convex boundary of $P_n$, the $p_j$ cannot be among these, so me must have $p_j = d_{i_j}$.
Let $i_{j} < i_{j+1}$ so that these points are in counterclockwise order.

We cannot have all the points $q_i$ on the same side of $\overline{p_jp_{j+1}}$.
Furthermore, none of the points of $P_n$ except $d_j$ are in the cone $d_j +\cone(d_j-d_{j-1},d_j-d_{j+1})$. 
Consequently, none of the point of $P_n$ that are to the right of $\overline{p_jp_{j+1}}$ are also to the right of $\overline{p_{j-1}p_{j}}$.  
Moreover, for each pair of distinct $j,k$, there are no points to the right of both $\overline{p_jp_{j+1}}$ and $\overline{p_kp_{k+1}}$. 
Hence, for each $j$ there is only one unique point among the $q_i$ that is to the right of $\overline{p_jp_{j+1}}$, which we will denote by $q_j$.
To have each point $p_j$ in the convex hull of the points $q_i$, we must have that each $p_j$ is to the left of $\overline{q_{j-1}q_j}$.

\begin{figure}[ht]
\centering
\includegraphics[scale=0.8]{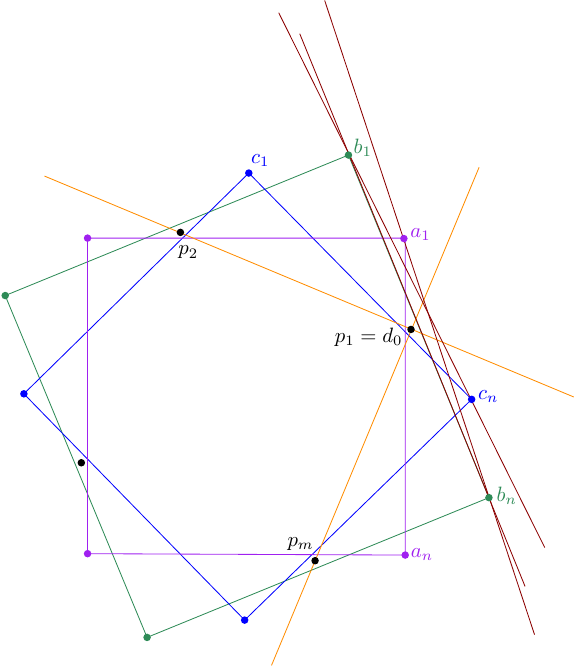}
\caption{The only lines from a point that is to the right of $\overline{p_mp_1}$ to a point that is to the right of $\overline{p_1p_2}$ with the point $p_1$ on the left are the lines $\overline{b_nb_1}$, $\overline{b_na_1}$, and $\overline{c_nb_1}$.  Hence, $q_m \in \{b_n,c_n\}$ and $q_1 \in \{a_1,b_1\}$.}
\label{figLinesLeftOfP1}
\end{figure}

To order the indices of the points of $P_n$ in a convenient way, 
let us relabel $d_n$ as $d_0$, and 
let us assume that $i_1 = 0$ and $p_1= d_0$; otherwise we may reindex the points of $P_n$ appropriately. 
Then, $i_m > i_2$, and $q_m$ is to the right of $\overline{p_mp_1}$, so $q_m$ is among the points $a_i,b_i,c_i$ for $i > i_m$, and $q_1$ is to the right of $\overline{p_1p_2}$, so $q_1$ is among the points $a_i,b_i,c_i$ for $i \leq i_2$; see Figure \ref{figLinesLeftOfP1}.  In particular, $q_m$ is counterclockwise from $q_2$ on the counterclockwise arc of the convex boundary of $P_n$ from $a_1$ to $c_n$. 
If $q_m$ were among points $a_{i_m},b_{i_m},\dots,a_n$, then $p_1 = d_0 = d_n$ would be to the right of $\overline{q_mq_1}$, since $d_n$ is to the right of $\overline{a_na_1}$ and $q_1$ is on the counterclockwise arc of the the convex boundary of $P_n$ from $a_1$ to $q_m$.
Hence, $q_m \in \{b_n,c_n\}$ and similarly $q_1 \in \{a_1,b_1\}$. 
Likewise, $q_{j-1} \in \{b_{i_j},c_{i_j}\}$ and $q_{j} \in \{a_{i_j+1},b_{i_j+1}\}$ for each $j \in \{1,\dots,m\}$ where indices $j$ of $q_j$ are considered mod $m$ and indices $i$ of $a_i,b_i,c_i$ are considered mod $n$. 
Therefore, each $q_j \in \{a_{i_j+1},b_{i_j+1}\} \cap \{b_{i_{j+1}},c_{i_{j+1}}\}$, so $q_j = b_j$ and $i_{j+1} = i_j+1$ mod $n$, so $i_2 = 1$, $i_3 = 2$,$\dots$, $i_m = m-1$, and $i_1 = m$ mod $n$, but we already have $i_1 = 0$ and $n > m$, so no such points $q_1,\dots,q_m$ can be found among $P_n$.
\end{proof}

\subsection{More uninscribable order types}\label{subsection_k-star}
In this section, we generalize the uninscribable point configuration $P_n$. Before doing so,  let us recall the construction of $P_n$. The essential property which makes $P_n$ uninscribable is that
\begin{itemize}
    \item[(S)] the orientations of triples $(a_{k-1}, a_k, a_{k+1})$, $(b_{k-1}, b_k, b_{k+1})$ and $(c_{k-1}, c_k, c_{k+1})$ are all positive, that is, the triples are ordered counterclockwise, and
    
    \item[(L)] the point $d_k$ is chosen in the triangle formed by $L_k$, $\Line{a_ka_{k+1}}$ and $\Line{c_kc_{k+1}}$ where the orientation of the triple $(b_k,b_{k+1},d_k)$ is positive, that is, $d_k$ is on the left to the oriented line $\Line{b_kb_{k+1}}$
\end{itemize}
 for every $k\in [n]$ where $k-1, k$ and $k+1$ are taken in modulo $n$. Using these properties, we define a broader collection of point configurations as follows.
 
\smallskip 
 
Let $a_1, b_1, c_1, \dots, a_n, b_n, c_n$ be the points in convex position which are in the same order as in $P_n$. For a bijection $\sigma : [n] \to [n]$, let $\mc{F}_\sigma$ be the family of configurations $\{a_k, b_k, c_k, d_k: k\in [n]\}$ that satisfies the following. 
\begin{itemize}
\item[(L$'$)] $d_{\sigma(k)}$ is chosen in the triangle formed by $\Line{a_{\sigma(k)}a_{\sigma(k+1)}}$, $\Line{b_{\sigma(k)}b_{\sigma(k+1)}}$ and $\Line{c_{\sigma(k)}c_{\sigma(k+1)}}$ where the orientation of the triple $(b_{\sigma(k)},b_{\sigma(k+1)},d_{\sigma(k)})$ is positive for every $k\in [n]$. 
 \end{itemize}
In particular, a configuration in $\mc{F}_\sigma$ is called a \df{star configuration} if $\sigma$ satisfies the following.
\begin{itemize}
    \item[(S$'$)] Each consecutive triple defines a permutation of positive sign, that is, $\sign(\sigma(k),\sigma(k+1),\sigma(k+2)) := (-1)^\iota = 1$ for each $k \in [n]$ with addition mod $n$ 
where $\iota = |\{\{i,j\}\subset \{k,k+1,k+2\} : i<j, \sigma(i)>\sigma(j)\}|$ is the number of inversions.
\end{itemize}
 Note that $P_n$ is a star configuration in  $\mc{F}_{\id_n}$ where $\id_n : [n] \to [n]$ is the identity. 

\smallskip
The proof of Lemma \ref{LemmaTwoPolygons} similarly applies to star configurations with a few modifications with respect to $\sigma$. To be more precise, we define a function $f_k^\sigma$ similarly with $f_k$ as the projection with respect to $p_{\sigma(k)}$, where $p_{\sigma(k)}$ are points chosen inside the circle. As the proof  assumed that $G_a$ and $G_c$ contain $p_k$ on each edge, we similarly impose the condition that, $a_1, c_1, \dots, a_n, c_n$ are ordered counterclockwise on the circle $\mb{S}^1$, and $p_{\sigma(k)}$
is the intersection point of the lines $\Line{a_{\sigma(k)}a_{\sigma(k+1)}}$ and $\Line{c_{\sigma(k)}c_{\sigma(k+1)}}$. Then, we define $g^\sigma:\overline{\mb{C}}\to \overline{\mb{C}}$ as the linear fractional transformation such that $g^{\sigma}(a_{\sigma(1)})=0$, $g^{\sigma}(c_{\sigma(1)})=\infty$, and $g^{\sigma}(x_{\sigma(1)})=1$ where $x_{\sigma(1)}$ is the midpoint of the counterclockwise arc from $a_{\sigma(1)}$ to $c_{\sigma(1)}$. Finally, let $f^\sigma=f^\sigma_{\sigma(n)}\circ\cdots \circ f^\sigma_{\sigma(1)}$ and $h^\sigma =g^\sigma \circ f^\sigma \circ (g^\sigma)^{-1}$.

The first part, with a slightly different choice of the initial point at the counterclockwise arc from $c_{\sigma(2)}$ to $a_{\sigma(n)}$, similarly holds because of (S$'$). The second part should be interpreted that whenever
we rotate and move each line $b_{\sigma(k)}b_{\sigma(k+1)}$ appropriately so that the $p_{\sigma(k)}$ is on $b_{\sigma(k)}b_{\sigma(k+1)}$, the final point on the circle where we arrive by following a chain of the resulting lines starting from $b_{\sigma(1)}$ should be located counterclockwise of $b_{\sigma(1)}$ in the counterclockwise arc from $a_{\sigma(1)}$ to $c_{\sigma(1)}$, which leads to a contradiction with the first part and yields the conclusion. This holds  because of (L$'$). Therefore, we can conclude that a star configuration is uninscribable, and by Lemma \ref{lemma_uninscribable_circle_conic}, we have the following.
\begin{theorem}\label{theorem_k-star}
No star configurations can be realized with the extreme points on a conic.
\end{theorem}
It does not look like Theorem \ref{theorem_k-star} can be easily generalized for $\mc{F}_\sigma$ with arbitrary $\sigma$.

\section{Order types with few interior or extreme points} \label{Section_few_points}

In this section, we show that an order type $\omega$ is inscribable if $\omega$ has at most 2 interior points (Theorem \ref{theorem_two_interior_inscribable}) or has at most 5 extreme points (Proposition \ref{prop_(5,3)_inscribable}).  In Subsection \ref{sectionFewExtremePoints} we briefly discuss the case with few extreme points and prove Proposition \ref{prop_(5,3)_inscribable} and Lemma \ref{lemma_uninscribable_circle_conic}.
We prove Theorem \ref{theorem_two_interior_inscribable} in Subsection \ref{subsection_two_interior}. As a consequence of Theorem \ref{theorem_two_interior_inscribable} and its proof, we obtain an estimate of the number of simple order types with two interior points in Subsection \ref{subsection_counting}.

\subsection{Simple order types with at most 5 extreme points}
\label{sectionFewExtremePoints}

\label{subsection 3 interior in petagon}
In this subsection, we consider order types with at most 5 extreme points. The proof of the following lemma uses a suitable projective transformation. For a detailed explanation on projective transformations and their effects on conics, one can consult \cite{projective_geometry_book}.

\begin{lemma} \label{lemma_uninscribable_circle_conic}
	If an order type $\omega$ is uninscribable, then $\omega$ cannot be realized as a point set where the extreme points are chosen from a conic.
\end{lemma}

\begin{proof}
	For a contradiction, suppose that $\omega$ is realized as a point set $P$ where the extreme points are chosen from a conic $\gamma$. When $\gamma$ is either an ellipse or a parabola, we can choose a projective transformation which sends $\gamma$ to $\mb{S}^1$ and preserves the order type of $P$. We can also choose a suitable projective transformation when $\gamma$ is a hyperbola and all extreme points of $P$ are from the same component of $\gamma$. So we are done in these cases. Hence, suppose that $\gamma$ is a hyperbola and each component of $\gamma$ contains at least one extreme point of $P$. Let us denote the components of $\gamma$ by $\gamma_1$ and $\gamma_2$.
	
	For distinct $i,j \in [2]$, note that we cannot have 3 extreme points $p_1$, $p_2$ and $p_3$ of $P'$ which appear in order on $\gamma_i$ and another extreme point $q$ on $\gamma_j$, since $p_2 \in \conv \{p_1, p_3, q\}$ so $p_2$ is not extreme any more. This requires that the number $k$ of extreme points of $P$ is at most 4. If $k=3$, we can find a circle which passes through the 3 points. If the number is $k=4$, we find an additional point which forms, together with the extreme points of $P$, a point set $B$ in convex position. By using the same argument, if the unique conic which passes through all points in $B$ is a hyperbola, then $B$ should be on the same component of it. So, again, we can use a suitable projective transformation to conclude the proof.
\end{proof}

From Lemma \ref{lemma_uninscribable_circle_conic}, the following inscribability claim is straightforward to obtain.

\begin{proof}[Proof of Proposition \ref{prop_(5,3)_inscribable}]
	Let $P$ be a point set which realizes $\omega$. Since
	there are at most 5 extreme points in $P$, we can find a conic which passes through all the extreme points. So by Lemma \ref{lemma_uninscribable_circle_conic}, $\omega$ is inscribable.
\end{proof}

\subsection{Inscribing simple order types with at most 2 interior points}
\label{subsection_two_interior}

In this section, we prove Theorem  \ref{theorem_two_interior_inscribable}. The main ingredient of the proof is a function representation of a simple order type with 2 interior points, which we define below. This representation can be regarded as an extension of a classical correspondence between the order types of conowheels (i.e. point sets with at most 1 interior point) and 2-colored self-dual necklaces. For the order types of conowheels, the inscribability can be easily shown by using the  corresponding self-dual necklaces, see the proof of Theorem \ref{theorem_two_interior_inscribable} for  details. With a similar spirit, we show that for a simple order type $\omega$ with 2 interior points there is another inscribable order type $\omega'$ that shares the same function representation with $\omega$, which implies $\omega=\omega'$.

\begin{figure}[ht]
	\centering
	\includegraphics{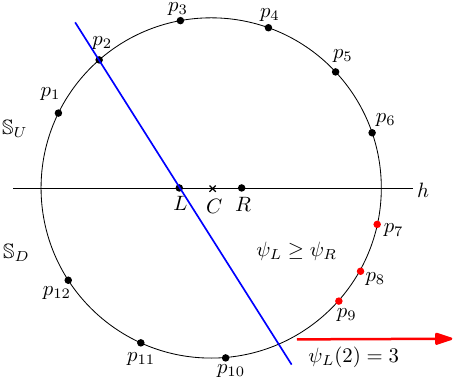}
	\caption{An example of a function representation. In this example, $n_1=n_2=6$.}
	\label{fig_2_int_realizability}
\end{figure}

Given a point set $P$ in general position with exactly two interior points, we assign a pair of functions $(\psi_L,\psi_R)$ to $P$ as follows; see Figure \ref{fig_2_int_realizability}. 
Let the pair $(L,R)$ be the interior points of $P$.
We can assume that the directed line $h=\Line{LR}$ is horizontal, where $L$ is left of $R$. Let $n_1$ (and $n_2$) respectively denote the number of points in $P$ strictly above (and below) $h$, and let $n=n_1+n_2$. 

Choose a point $C$ in $\conv(P) \cap h$ between $L$ and $R$, and label the points in $P \setminus \{L, R\}$ by $p_1, \dots, p_n$ in clockwise order around $C$ starting from $L$. Finally, we define a function $\psi_A:[n_1] \to [n_2]\cup \{0\}$ for $A\in \{L,R\}$ where $\psi_A(i)$ is the number of the points in $P$ strictly below $h$ and to the right of the directed line $\Line{Ap_i}$.  
We say that the pair of functions $(\psi_L, \psi_R)$ is the \textit{function representation} of $P$ with respect to the ordered pair $(L,R)$. 

Note that a function representation is completely determined by the order type of a given point set and the choice of the pair $(L,R)$. Therefore, there are at most two different function representations for each order type $\omega$ with exactly 2 interior points.


\begin{prop} \label{prop_function_determines_ot}
Let $\omega$ and $\omega'$ be simple order types with 2 interior points. If $\omega$ and $\omega'$ share a common function representation $(\psi_L, \psi_R)$, then $\omega=\omega'$.
\end{prop}
\begin{proof}
	Let $P$ and $P'$ be realizations of $\omega$ and $\omega'$, and let $(L,R)$ and $(L',R')$ be ordered pairs of interior points of $P$ and $P'$, repsectively, which are used to define $(\psi_L, \psi_R)$. As we did in the above, we assume $h=\overline{LR}$ (or $h'=\overline{L'R'}$) is horizontal and $L$ (or $L'$) is left to $R$ (or $R'$, respectively). Also similarly with the above, we order clockwise the extreme points of $P$ as $p_1, \dots, p_n$ with respect to $(L,R)$. We also order the extreme points of $P'$ as $p_1', \dots, p_n'$ with respect to $(L',R')$. We claim that the bijection $\Phi: P \to P'$ defined as $\Phi(L)=L'$, $\Phi(R)=R'$, and $\Phi(p_i)=p_i'$ for every $i\in [n]$ preserves the orientation of every triple in $P$. 
	
	The claim obviously holds when a triple of $P$ contains all or none of $L$ and $R$. So, without loss of generality, we assume that the triple contains $L$ but not $R$, and two other points $p_i$ and $p_j$ where $i<j$. There are two cases to consider.
	
	\smallskip
	
	(Case 1) \df{When $p_i$ and $p_j$ are separated by $h$}: By our setting, $p_i$ is above $h$ and $p_j$ is below $h$, and a similar situation holds for $p_i'$, $p_j'$ and $h'$. It is enough to show that $p_j$ is on the right to the directed line $\overline{Lp_i}$ if and only if $p_j'$ is on the right to the directed line $\overline{L'p_i'}$. By the definition of function representations, the former and the latter condition is equivalent to that $\psi_L(i) \geq j-n_1$, where $n_1$ is the common size of the domain of $\psi_L$ and $\psi_R$ which is the same as the number of extreme points of $P$ (or $P'$) above $h$ (or $h'$, respectively). So the orientation is preserved in this case.
	
	\smallskip 
	
	(Case 2) \df{When $p_i$ and $p_j$ are at the same side of $h$}: 
	We only consider the case when $p_i$ and $p_j$ are both above $h$, and the other case can be argued similarly. 
	Note that $\conv(P) \cap h$ is on the right side of the directed line $\overline{p_ip_j}$. In particular, this implies $L$ should be also on the right to $\overline{p_ip_j}$. Since the same happens with $p_i'$, $p_j'$, $L'$ in $P'$, the orientation is preserved in this case.
\end{proof}

\medskip 

Suppose that a simple order type with two interior points has a function representation $(\psi_L, \psi_R)$ with the domain $[n_1]$ and the codomain $[n_2]\cup \{0\}$. It is easy to check that $(\psi_L, \psi_R)$ satisfies
\begin{itemize}
	\item[(i)] $n_1 \geq 1$ and $n_2 \geq 1$. \label{con_not_vacuous}
\end{itemize}
In particular, (i) implies the domain is non-empty. $(\psi_L, \psi_R)$ also has the following properties. The verification is left to the reader.

\begin{itemize}
	\item[(ii)] $\psi_L$ and $\psi_R$ are monotonically increasing. \label{con_increasing}
	\item[(iii)] $\psi_L(i)\geq \psi_R(i)$ for every $i\in [n_1]$. \label{con_L_large}
	\item[(iv)] $\psi_L \not\equiv n_2$ and $\psi_R \not\equiv 0$. \label{con_not_constant}
\end{itemize}

The following theorem shows that the converse is true even when restricted to inscribable order types.

\begin{theorem} \label{theorem_two_functions_inscribable}
Let $(\psi_L, \psi_R)$ be a pair of functions with the domain $[n_1]$ and the codomain $[n_2]\cup \{0\}$. Suppose that $(\psi_L, \psi_R)$ satisfy the above conditions (i), (ii), (iii) and (iv). Then there is a simple inscribable order type with 2 interior points that has $(\psi_L, \psi_R)$ as its function representation.
\end{theorem}

\begin{proof}[Proof of Theorem \ref{theorem_two_interior_inscribable}]
For the case of one interior point,  
Pilz, Welzl, and Wettstein gave a bijection $\Psi$ from the set of order types of conowheels of size $n+1$ to the set of 2-colored self-dual necklaces with $2n$ beads \cite[Theorem 1.1]{one_point_pilz_welzl}, and their construction shows that an image of $\Psi^{-1}$ is inscribable. 
	
Now let us consider the 2 interior point case. For a simple order type $\omega$ with 2 interior points, we find a function representation $(\psi_L, \psi_R)$. Since $(\psi_L, \psi_R)$ satisfies the conditions (i), (ii), (iii) and (iv), 
by Theorem \ref{theorem_two_functions_inscribable} there is a simple inscribable order type $\omega'$ with 2 interior points which has a function representation $(\psi_L, \psi_R)$. By Proposition \ref{prop_function_determines_ot}, $\omega= \omega'$. This implies $\omega$ inscribable. 
\end{proof}

We introduce some notation before we prove Theorem \ref{theorem_two_functions_inscribable}.
For a directed line $\overline{AB}$ which is not horizontal,
let $\rhalf{A,B}$ (or $\lhalf{A,B}$) be the open half-plane determined by $\overline{AB}$, which is right (or left, respectively) of $\overline{AB}$. For a point $p$ on the unit circle $\mb{S}^1$ and another point $A$ with $||A||<1$, let $p^A \neq A$ be the other intersection point of $\Line{pA}$ and $\mb{S}^1$.

\begin{proof}[Proof of Theorem \ref{theorem_two_functions_inscribable}]
	Let $(\psi_L, \psi_R)$ be a pair of functions which satisfy the conditions. Let $L=(-1/2,0)$ and $R=(1/2,0)$. Let $\mb{S}_U$ and $\mb{S}_D$ be the open semicircles in the unit circle $\mb{S}^1$ which are above and below $\overline{LR}$, respectively.

	We first inductively construct $P_i$ for $i\in [n_1]$ (note that $n_1 \geq 1$ by (i)). 
	Each $P_i$ is supposed to have the following property which depends on the index $i$.
	\begin{itemize}
		\item[(F)] $(\psi_L|_{[i]}, \psi_R|_{[i]})$ is the function representation of $P_i$ with respect to $(L,R)$  where the codomain of the functions are restricted to $[\psi_L(i)]\cup \{0\}$.
	\end{itemize}

	\smallskip 
	
 For the basis step, we choose $p_1$ arbitrarily from $\mb{S}_U$. Then we choose points in the clockwise order $q_1, \dots, q_{\psi_R(1)}$ from $\rhalf{R,p_1}\cap \mb{S}_D$. By (iii), $\psi_L(1)\geq \psi_R(1)$. When $\psi_L(1)= \psi_R(1)$, we stop here. If $\psi_L(1)> \psi_R(1)$, then we choose in the clockwise order $q_{\psi_R(1)+1}, \dots, q_{\psi_L(1)}$ from  $\rhalf{L, p_1}\cap \lhalf{R,p_1}\cap \mb{S}_D$. By the construction the point set $P_1=\{p_1, q_1, \dots, q_{\psi_L(1)}, L, R\}$ satisfies (F).
	
	\smallskip
	
	Now, suppose that a point set
	$$P_{i-1}=\{p_1, \dots, p_{i-1}\}\cup \{q_1, \dots, q_{\psi_L(i-1)}\}\cup\{L,R\} $$
	is already constructed and satisfies (F), where $p_1, \dots, p_{i-1}$ are points on $\mb{S}_U$ and $q_1, \dots, q_{\psi_L(i-1)}$ are points on $\mb{S}_D$. These points are all ordered clockwise. We extend $P_{i-1}$ to 	$$P_i= \{p_1, \dots, p_i\}\cup \{q_1, \dots, q_{\psi_L(i)}\}\cup\{L,R\}$$ by suitably choosing $p_i$ in $\mb{S}_U$ and $q_{\psi_L(i-1)+1}, \dots, q_{\psi_L(i)}$ in $\mb{S}_D$. There are two cases (see Figure \ref{fig_2_int_inscribability}). 
	\begin{figure}[ht]
		\centering
		\includegraphics{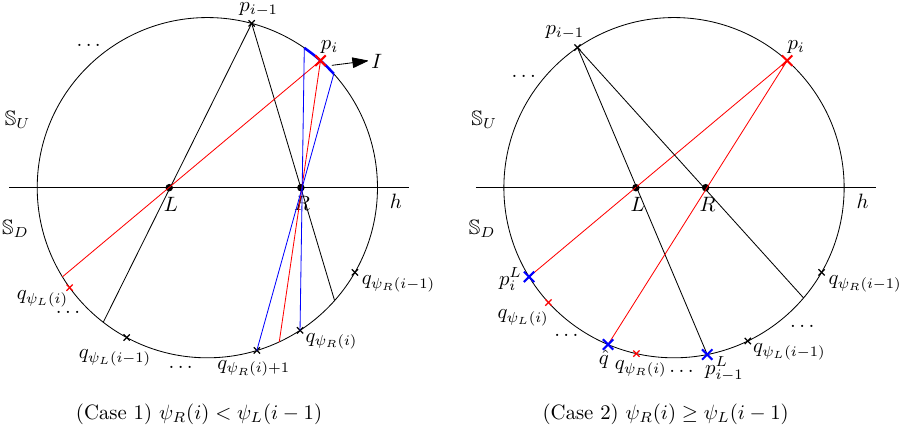}
		\caption{The inductive step in the proof of Theorem \ref{theorem_two_functions_inscribable}.}
		\label{fig_2_int_inscribability}
	\end{figure}
	
	\smallskip 
	
	(Case 1) \df{When $\psi_R(i)< \psi_L(i-1)$}:  Let $I$ be a closed arc in $\mb{S}_U$ with $q_{\psi_R(i)}^R$ and $q_{\psi_R(i)+1}^R$ as its boundaries. If $p_{i-1}$ is positioned right to $I$, then we have more than $\psi_R(i)$ points in $\{q_1, \dots, q_{\psi_L(i-1)}\} \cap \rhalf{R,p_{i-1}}$. Since  $\psi_R(i)\geq \psi_R(i-1)$ by (ii), this leads to a contradiction to the induction hypothesis. Hence, $p_{i-1}$ is either in $I$ or left to $I$. In either case, we can choose $p_i$ from $I$ to the right of $p_{i-1}$.  Then, the reflected point $p_i^L$ is to the left of $p_{i-1}^L$ on $\mb{S}_D$. Note that $\psi_L(i)\geq \psi_L(i-1)$ by (ii). When $\psi_L(i)> \psi_L(i-1)$, we choose points $q_{\psi_L(i-1)+1}, \dots, q_{\psi_L(i)}$ between $p_{i-1}^L$ and $p_i^L$ from $\mb{S}_D$ in the clockwise order. This completes the construction of $P_i$. By construction, 
	$P_i$ satisfies the condition (F).
	
	\smallskip
	
	(Case 2) \df{When $\psi_R(i)\geq \psi_L(i-1)$}: We choose a point $\hat{q} \in \mb{S}_D$ which is left to $p_{i-1}^L$, and let $p_i:=\hat{q}^R \in \mb{S}_U$. We choose $q_{\psi_L(i-1)+1}, \dots, q_{\psi_R(i)}$ on $\mb{S}_D$ in the clockwise order between $p_{i-1}^L$ and $\hat{q}$ if $\psi_R(i)>\psi_L(i-1)$. Note that $p_i^L$ is left to $\hat{q}$. By (iii), we have $\psi_L(i)\geq \psi_R(i)$. If $\psi_L(i)> \psi_R(i)$, we choose $q_{\psi_R(i)+1}, \dots, q_{\psi_L(i)}$ on $\mb{S}_D$ between $\hat{q}$ and $p_i^L$. Again, 
	the above procedure gives $P_i$ which satisfies the condition (F).
	
	\smallskip
	This inductive procedure yields a point set $P_{n_1}$ which satisfies (F) where $p_{n_1}$ is the rightmost point on $\mb{S}_U \cap P_{n_1}$. 
	If $n_2>\psi_L(n_1)$, we additionally put $q_{\psi_L(n_1)+1}, \dots, q_{n_2}$ in the clockwise order on $\mb{S}_D$, on the left to $p_{n_1}^L$. Adding these additional points to $P_{n_1}$ gives the point set $P$, and the function representation of $P$ with repsect to $(L,R)$ is exactly $(\psi_L, \psi_R)$. One can easily check that $P$ is in general position by construction.

\smallskip

It remains to show that $L$ and $R$ are interior points. 
By (iv), there are points $p_i$ and $p_j$ such that $\psi_L(i) < n_2$ and $\psi_R(j) > 0$, 
so $L \in \conv(\{p_i,p_n,R\})$ and $R \in \conv(\{p_j,p_{n_1+1},L\})$. 
Thus, $L$ and $R$ are interior points.
\end{proof}

\begin{remark} \label{remark_convex_curve}
We can apply the same construction of Theorem \ref{theorem_two_functions_inscribable} to the boundary of a convex body $C$ if $\partial C$ does not contain a line segment. Also, $L$ and $R$ can be chosen arbitrarily among distinct pairs of interior points of $C$.
\end{remark}

\subsection{Counting simple order types with at most 2 interior points} \label{subsection_counting}

In this section, we give an estimate of the number of simple order types with at most 2 interior points.

For nonnegative integers $n$ and $i$, let $\mathcal{O}_i(n)$ be the set of simple order types of size $n+2$ with exactly $i$ interior points. Let $\mathcal{F}(n_1, n_2)$ be the set of pairs of functions $(\psi_L, \psi_R)$ from the domain $[n_1]$ to the codomain $[n_2]\cup \{0\}$ which satisfy the conditions (i)-(iv) from Subsection \ref{subsection_two_interior}. Note that by the condition (i), both $\mathcal{F}(0,n_2)$ and $\mathcal{F}(n_1,0)$ are empty. Also, let 
\[\mathcal{F}(n):=\bigcup_{m=0}^{n} \mathcal{F}(m,n-m).\]

In Subsection \ref{subsection_two_interior} we showed that each element of $\mathcal{F}(n)$ defines a unique order type, but they might not define distinct order types. 
The function representation for a given order type depends only on the choice of which interior point is $L$ and which is $R$.  Swapping $L$ and $R$ changes the function representation for some order types, but not for others.  Consequently, we have the following.

\begin{corollary} \label{corollary_counting}
$|\mathcal{F}(n)|/2 \leq |\mathcal{O}_2(n)| \leq |\mathcal{F}(n)|$. In particular, $|\mathcal{O}_2(n)| \in \Theta(|\mathcal{F}(n)|)$. 
\end{corollary}
\begin{proof}
By Theorem \ref{theorem_two_functions_inscribable}, every pair of functions in $\mathcal{F}(n)$ is obtained as a function representation of an order type in $\mathcal{O}_2(n)$. By Proposition \ref{prop_function_determines_ot}, the order type is determined uniquely for every pair in $\mathcal{F}(n)$. So, the relation $\sim$ on $\mathcal{F}(n)$ defined by that $(f_L, f_R) \sim (g_L, g_R)$ if $(f_L, f_R)$ and $(g_L, g_R)$ are function representations of the same order type is an equivalence relation. Since every order type in $\mathcal{O}_2$ has a function representation in $\mathcal{F}(n)$, there is a 1-1 correspondence between $\mathcal{O}_2(n)$ and $\mathcal{F}/\sim$. The conclusion comes from that every equivalence class in $\mathcal{F}/\sim$ has size either 1 or 2. 
\end{proof}

\medskip

A \textit{plane partition} is an array $\pi=(\pi_{i,j})_{i,j \geq 1}$ of nonnegative integers that has finite support and is weakly decreasing in rows and columns. That is, $\pi$ has only finitely many nonzero entries and $\pi_{i, j}\geq \pi_{i, j+1}$ and $\pi_{i, j}\geq \pi_{i+1, j}$. For nonnegative integers $a$, $b$ and $c$, let $\mathcal{P}(a,b,c)$ be the set of all plane partitions $\pi$ such that the support of $\pi$ is contained in an index subset $[a]\times [b]$ and the entries of $\pi$ are bounded above by $c$. 
Note that plane partitions correspond to lozenge tilings, but we will not make use of that correspondence here \cite{handbook_enumerative_combinatorics}.
The following counting result is known.

\begin{theorem}[MacMahon \cite{macmahon1916combinatory}, 
see also p.545 of \cite{handbook_enumerative_combinatorics}, or p.378 of \cite{stanley1999enumerative} for an equivalent formula]
\label{theorem_logenze}
	\begin{align*}
	|\mathcal{P}(a,b,c)|=\frac{H(a + b + c)H(a)H(b)H(c)}{H(a + b)H(a + c)H(b + c)}
	\end{align*}
	where $H(0) = H(1) = 1$ and $H(n) = 1!2!\cdots(n-1)!$ for $n > 1$.
\end{theorem}
	
Let $\pi : \mathcal{F}(n_1, n_2) \to \mathcal{P}(2, n_1, n_2)$ by 
$\pi_{1, j} (\psi_L, \psi_R)=\psi_L(n_1-j+1)$ and $\pi_{2, j} (\psi_L, \psi_R)=\psi_R(n_1-j+1)$, 
and observe that this map is injective.
Moreover, pairs of functions in $\mathcal{F}(n_1, n_2)$ exactly correspond to plane partitions in
$$\mathcal{P}(2,n_1,n_2)\setminus (\mathcal{M}(2,n_1,n_2) \cup \mathcal{Z}(2,n_1, n_2)),$$
where
\begin{align*}
\mathcal{M}(2,n_1,n_2)&= \{\pi \in \mathcal{P}(2,n_1, n_2):\pi_{1,j}=n_2 \textrm{ for every $1 \leq j \leq n_1$} \},\textrm{ and }\\
\mathcal{Z}(2,n_1, n_2)&=\{\pi \in \mathcal{P}(2,n_1, n_2):\pi_{2,j}=0 \textrm{ for every $j\geq 1$} \}.
\end{align*}
Note that when $n_1=0$ or $n_2=0$, $|\mathcal{P}(2,n_1,n_2)|= |\mathcal{Z}(2,n_1, n_2)| = |\mathcal{M}(2,n_1,n_2)|=1$ where the only element of these sets is the array whose entries are all 0. We also have $|\mathcal{Z}(2,n_1, n_2) \cap  \mathcal{M}(2,n_1,n_2)|=1$ for every nonnegative integers $n_1$ and $n_2$. Therefore,
\begin{align}
&|\mathcal{F}(n)|\notag=\sum_{m=0}^n |\mathcal{F}(m,n-m)|\\
=&\sum_{m=0}^n |\mathcal{P}(2,m,n-m)\setminus (\mathcal{M}(2,m,n-m) \cup \mathcal{Z}(2,m, n-m))| \notag \\
=&\sum_{m=0}^n (|\mathcal{P}(2,m,n-m)|- (|\mathcal{M}(2,m,n-m)|+ |\mathcal{Z}(2,m, n-m)|-1)) \notag\\
=&\sum_{m=0}^{n}|\mathcal{P}(2,m,n-m)|-\sum_{m=0}^n|\mathcal{M}(2,m,n-m)|- \sum_{m=0}^n|\mathcal{Z}(2,m,n-m)|+(n+1). \label{eq_exact}
\end{align}
Using Theorem \ref{theorem_logenze}, the first term in (\ref{eq_exact}) is 
\begin{align}
&\sum_{m=0}^{n}|\mathcal{P}(2,m,n-m)|\notag\\
=&\sum_{m=0}^{n} \frac{H(n+2)H(m)H(n-m)H(2)}{H(n)H(m+2)H(n-m+2)}
=\sum_{m=0}^{n} \frac{(n+1)!n!}{m!(m+1)!(n-m)!(n-m+1)!}\notag\\
=&\frac{1}{n+1}\sum_{m=0}^{n} \frac{(n+1)!}{(m+1)!(n-m)!}\cdot\frac{(n+1)!}{m!(n-m+1)!}
\notag\\
=&\frac{1}{n+1}\sum_{m=0}^{n} \binom{n+1}{m}\binom{n+1}{n-m}
=\frac{1}{n+1}\binom{2n+2}{n}\notag\\
=&\frac{1}{n+1}\cdot\frac{(2n+2)!}{n!(n+2)!}
=\frac{(2n+2)(2n+1)}{(n+1)^2(n+2)} \binom{2n}{n}
\sim \frac{4}{n}\cdot \frac{4^n}{\sqrt{\pi n}}
=\frac{4^{n+1}}{\sqrt{\pi}n^{3/2}}, \label{eq_asymptote}
\end{align}
where the asymptote $\binom{2n}{n}\sim \frac{4^n}{\sqrt{\pi n}}$ can be obtained from Stirling's formula. Also, by only reading off the non-constant row, the  summations which appear at the second and third terms in (\ref{eq_exact}) are same as
\begin{align}
&\sum_{m=0}^{n}|\mathcal{P}(1,m,n-m)|\notag\\
=&\sum_{m=0}^{n} \frac{H(n+1)H(m)H(n-m)}{H(n)H(m+1)H(n-m+1)}
=\sum_{m=0}^{n} \frac{n!}{m!(n-m)!}=2^n.\notag
\end{align}
Therefore the first term of (\ref{eq_exact}) is the leading term, and by (\ref{eq_asymptote}) we have  $|\mathcal{F}(n)| \in \Theta(\frac{4^{n}}{n^{3/2}})$ which implies $|\mathcal{O}_2(n)| \in \Theta(\frac{4^{n}}{n^{3/2}})$ by Corollary \ref{corollary_counting}.
\smallskip

The following counting result is known for order types of conowheels. The proof uses a 1-1 correspondence between order types of conowheels and 2-colored self-dual necklaces.
\begin{theorem}[\cite{self_dual_necklace-brouwer, self_dual_necklace-palmer, one_point_pilz_welzl, grunbaum_polytope_book}] \label{theorem_conowheel}
	The number of order types of conowheels, that is, simple order types with at most 1 interior point of size $n+1$ is
	$$\frac{1}{4n}\sum_{2\nmid k|n} \varphi(k)2^{n/k}+2^{\lfloor (n-3)/2\rfloor}\in \Theta(2^n/n),$$
	where $\varphi$ is Euler’s totient function.
\end{theorem}

The above computation and Theorem \ref{theorem_conowheel} conclude the following.

\begin{theorem} \label{theorem_counting_2_interior}
	The number of simple order types with at most 2 interior points of size $n+2$ is $\Theta(\frac{4^n}{n^{3/2}})$.
\end{theorem}

\begin{remark}
In \cite{felsner_pseudoline_counting}, a similar correspondence between pseudoline arrangements and zonotopal tilings, which also correspond to plane partitions, is used to construct many pseudoline arrangements and count their number.
\end{remark}

\section{Concluding remarks}
\label{section_concluding}
\subsection{Realizing interior points on a circle}
A natural question which arises is whether an order type where the interior points are in convex position can be realized as a point set such that the interior points are on a circle. The configuration in Figure \ref{fig_internal_nonrealize} answers the question negatively.

\begin{figure}[ht]
\centering
\includegraphics{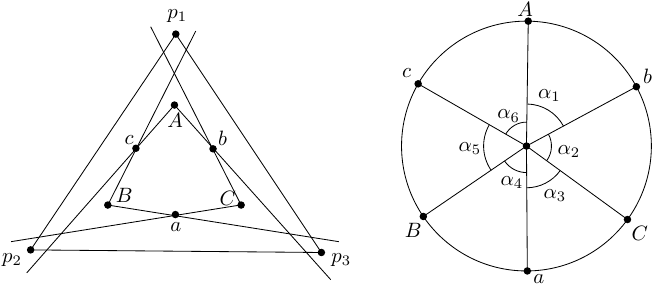}
\caption{Left: A configuration such that the interior points cannot be realized on a circle. Right: When the interior points are on a circle.}
\label{fig_internal_nonrealize}
\end{figure}
\begin{eg}
Let us denote the configuration on the left of Figure \ref{fig_internal_nonrealize} by $Q$. $Q$ has $A$, $B$, $C$, $a$, $b$ and $c$ as the interior points, and the $p_1$, $p_2$ and $p_3$ as extreme points. Note that $C$ is inside the triangle bounded by the lines $\Line{Ab}$, $\Line{Ba}$ and $\Line{AB}$. Similar properties also hold for $A$ and $B$, and their corresponding lines.

For a contradiction, suppose that the 6 interior points are on a circle as in the right of Figure \ref{fig_internal_nonrealize}. By the properties of $Q$, the angles $\alpha_1, \dots, \alpha_6$ satisfy
\begin{align*}
&\alpha_2+\alpha_3 < \alpha_5+\alpha_6,\\
&\alpha_1+\alpha_6 < \alpha_3+\alpha_4, \textrm{ and}\\
&\alpha_4+\alpha_5 < \alpha_1+\alpha_2.
\end{align*}
This implies $2\pi < 2\pi$, which leads to a contradiction. \qed
\end{eg}

\subsection{Characterization of minimally uninscribable order types}
Theorem \ref{theorem_main_minimally_uninscribable} shows that there are infinitely many minimally uninscribable order types. However, there might be many more minimally uninscribable order types.

\begin{Q}
What is the complete list of minimally uninscribable order types? Is there a non-trivial equivalent condition for being minimally uninscribable?
\end{Q}

Also, note that the theorems do not exclude the case when there are only finitely many minimally uninscribable order types with a fixed number of interior points.

\begin{Q}
Are there finitely many minimally uninscribable order types with exactly 3 interior points?
\end{Q}

\subsection{Abstract order types with few interior points}
Given an assignment $\chi$ of $\{+,0,-\}$ to each triple of a ground set, we say that $\chi$ is realizable when there is some point configuration in the plane where each triple has the assigned orientation.
\df{Abstract order types}, or equivalently \df{rank 3 acyclic chirotopes}, are assignments that satisfy  certain axioms, which are necessary combinatorial conditions for realizability \cite{om_book}. 
However, these axioms are not sufficient conditions for realizability. 
One reason for interest in abstract order types is that the problem of deciding realizability is computationally intractable.  In some situations we still want to work with orientation information like that captured by an order type, but we may not have any guarantee that the data came from an actual point set, such as in a computer program that works with orientation information provided by the user. 
It is easier to verify that a given assignment is an abstract order type. 
Theorems about point sets in the plane often only use the axioms of rank 3 acyclic chirotopes, and consequently the same theorem may hold more generally for abstract order types by the same proof as for order types.
Such is the case for the results in Subsection \ref{subsection_two_interior}.
Hence, we have the following.

\begin{theorem}
All simple abstract order types with at most 2 interior points are inscribable, and in particular are realizable.
\end{theorem}

There is a well-known nonrealizable abstract order type due to Ringel that has 4 interior points; see Figure \ref{fig_simple-non-Pappus}
(\cite{ringel_simple-non-Pappus, grunbaum_arrangements_spreads}, see also \cite[Section 8.3]{om_book}). 
This leaves us with the following question.

\begin{Q}
Are all simple abstract order types with exactly 3 interior points realizable?
\end{Q}

\begin{figure}[ht]
\centering
\includegraphics{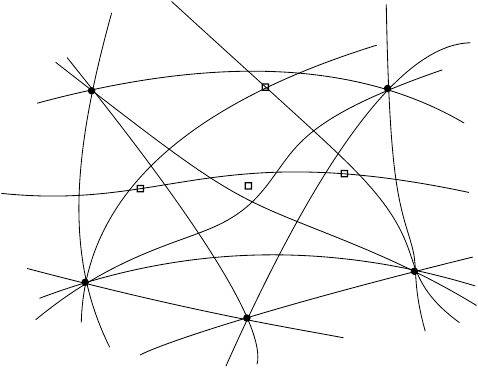}
\caption{Ringel's non-realizable abstract order type $\mathsf{Rin}(9)$. The points marked by boxes are interior points.}
\label{fig_simple-non-Pappus}
\end{figure}

\subsection{Counting order types with few interior points}

Here we have counted simple order types with 2 interior points asymptotically, but we could also consider the analogous question for the case of a fixed number of interior points.

\begin{Q}
For a fixed value $k \geq 3$, how many simple order types of size $n+k$ have $k$ interior points asymptotically? 
What about for simple abstract order types? 
\end{Q}
 
Goaoc and Welzl have shown that the expected number of extreme points of a simple order type chosen uniformly at random remains close to 4 as $n$ grows \cite{goaoc2023convex}. 
This suggests that the number simple order types with $k$ interior points should be much smaller than the total number of simple order types.

\section*{Acknowledgements}
We would like to thank Xavier Goaoc for bringing the problem of inscribability to our attention. We are also grateful to Jang Soo Kim for his helpful explanation about plane partitions.

\bibliographystyle{plain}
\bibliography{bibliography}
	
\end{document}